\documentclass[11pt,a4paper]{article}
\usepackage{amsfonts,amsgen,amstext,amsbsy,amsopn,amsfonts,amssymb,amscd}
\usepackage[leqno]{amsmath}
\usepackage[amsmath,amsthm,thmmarks]{ntheorem}
\usepackage{epsf,epsfig}
\usepackage{float}
\usepackage{ebezier,eepic}
\usepackage{color}
\usepackage{tikz}
\usepackage{multirow}
\usepackage{mathrsfs}
\usepackage{graphicx}
\usepackage{subfigure}
\newtheorem{thm}{Theorem}[section]

\newtheorem{prop}[thm]{Proposition}

\newtheorem{lem}[thm]{Lemma}
\newtheorem{false statement}{False statement}
\newtheorem{cor}[thm]{Corollary}
\newtheorem{fact}[thm]{Fact}

\theoremstyle{definition}

\newtheorem{claim}{Claim}

\newtheorem{example}[thm]{Example}

\makeatletter \@addtoreset{equation}{section}

\def\hh{\mathcal{H}}

\def\hf{\mathcal{F}}
\def\hg{\mathcal{G}}
\def\hk{\mathcal{K}}
\def\hl{\mathcal{L}}
\def\ha{\mathcal{A}}
\def\hb{\mathcal{B}}
\def\hs{\mathcal{S}}
\def\hht{\mathcal{T}}

\begin{document}
\title{\bf\Large A Product Version of the Hilton-Milner-Frankl Theorem}
\date{}
\author{Peter Frankl$^1$, Jian Wang$^2$\\[10pt]
$^{1}$R\'{e}nyi Institute, Budapest, Hungary\\[6pt]
$^{2}$Department of Mathematics,
Taiyuan University of Technology,\\
Taiyuan {\rm 030024}, China\\[6pt]
E-mail:  $^1$frankl.peter@renyi.hu, $^2$wangjian01@tyut.edu.cn
}
\maketitle

\begin{abstract}
Two families $\hf,\hg$ of $k$-subsets of $\{1,2,\ldots,n\}$ are called non-trivial
cross $t$-intersecting if $|F\cap G|\geq t$ for all $F\in \hf, G\in \hg$ and
$|\cap \{F\colon F\in \hf\}|<t$, $|\cap \{G\colon G\in\hg\}|<t$. In the present paper,
we determine the maximum product of the sizes of two non-trivial cross $t$-intersecting
families of   $k$-subsets of $\{1,2,\ldots,n\}$ for $n\geq 4(t+2)^2k^2$, $k\geq 5$,
which is a product version of the  Hilton-Milner-Frankl Theorem.

\vspace{6pt}

\vspace{6pt}
\end{abstract}

\section{Introduction}
Let $n>k>t$ be positive integers and let $[n]=\{1,2,\ldots,n\}$ be the standard $n$-element set. Let $\binom{[n]}{k}$ denote the collection of all $k$-subsets of $[n]$. Subsets of $\binom{[n]}{k}$ are called {\it $k$-uniform hypergraphs} or {\it $k$-graphs} for short.  A $k$-graph $\hf$ is called {\it $t$-intersecting} if $|F\cap F'|\geq t$ for all $F,F'\in \hf$.

One of the most important results in extremal set theory is the following:

\vspace{6pt}
{\noindent\bf Erd\H{o}s-Ko-Rado Theorem} (\cite{EKR}){\bf.} Suppose that $n\geq n_0(k,t)$ and $\hf\subset \binom{[n]}{k}$ is $t$-intersecting. Then
\begin{align}\label{ineq-ekr}
|\hf| \leq \binom{n-t}{k-t}.
\end{align}
\vspace{6pt}

{\noindent\bf Remark.} For $t=1$ the exact value $n_0(k,t)=(k-t+1)(t+1)$ was proved in \cite{EKR}. For
$t\geq 15$ it is due to \cite{F78}. Finally Wilson \cite{W84} closed the gap $2\leq t\leq 14$ with a proof valid for all $t$.

Pyber \cite{Pyber86} proved a product version of the Erd\H{o}s-Ko-Rado Theorem for $t=1$.

\begin{thm}[Pyber \cite{Pyber86}]
Suppose that $\hf,\hg\subset \binom{[n]}{k}$ are cross-intersecting, $n\geq 2k$ then
\begin{align}\label{ineq-pyber}
|\hf||\hg| \leq \binom{n-1}{k-1}^2.
\end{align}
\end{thm}

A $t$-intersecting family $\hf\subset\binom{[n]}{k}$ is called {\it non-trivial} if $|\cap \{F\colon F\in \hf\}|<t$.

\begin{example}
Define two  families
\begin{align*}
&\hh(n,k,t) =\left\{H\in \binom{[n]}{k}\colon [t]\subset H, H\cap [t+1,k+1]\neq \emptyset\right\}\cup \left\{[k+1]\setminus \{j\}\colon 1\leq j\leq t\right\},\\[5pt]
&\ha(n,k,t) =\left\{A\in \binom{[n]}{k}\colon |A\cap [t+2]|\geq t+1\right\}.
\end{align*}
It is easy to see that both $\hh(n,k,t)$ and $\ha(n,k,t)$ are non-trivial $t$-intersecting families.
\end{example}

\vspace{6pt}
{\noindent\bf Hilton-Milner-Frankl Theorem} (\cite{HM67,F78-2}){\bf.} Suppose that $\hf\subset\binom{[n]}{k}$ is non-trivial $t$-intersecting, $n\geq (k-t+1)(t+1)$. Then
\begin{align}\label{ineq-hmfrankl}
|\hf|\leq \max\left\{|\ha(n,k,t)|,|\hh(n,k,t)|\right\}.
\end{align}

Two families $\hf,\hg\subset \binom{[n]}{k}$ are called {\it cross $t$-intersecting} if $|F\cap G|\geq t$ for any $F\in \hf, G\in \hg$. If $\ha\subset \binom{[n]}{k}$ is $t$-intersecting, then $\hf=\ha$, $\hg=\ha$ are cross $t$-intersecting.

In the present paper, we prove a product version of the Hilton-Milner-Frankl Theorem.

\begin{thm}\label{thm-main}
Suppose that $\hf,\hg\subset \binom{[n]}{k}$ are non-trivial cross $t$-intersecting. If
$n\geq  4(t+2)^2k^2$ and $k\geq 5$, then
\begin{align}\label{nonemptycommon2}
|\hf||\hg|\leq \max\left\{|\ha(n,k,t)|^2,|\hh(n,k,t)|^2\right\}.
\end{align}
\end{thm}

It should be mentioned that the case $t=1$ of Theorem \ref{thm-main} was proved for $n\geq 9k$ and
$k\geq 6$ in \cite{FW2022}. Thus we always assume $t\geq 2$ in this paper. Since we are
concerned with the maximum of $|\hf||\hg|$, without further mention we are assuming throughout
the paper that $\hf$ and $\hg$ form a saturated pair, that is, adding an extra $k$-set to either
of the families would destroy the cross $t$-intersecting property.

It should also be mentioned that a similar result to Theorem \ref{thm-main} was obtained recently
 by Cao, Lu, Lv and Wang \cite{cllw2022}, in which a different notion of non-triviality was considered. In
 their paper, the non-triviality of two cross $t$-intersecting families $\hf,\hg$ means that
 $\hf\cup\hg$ is not a $t$-star.
However, to the best of our
knowledge \eqref{nonemptycommon2} is the first one that  implies the Hilton-Milner-Frankl Theorem (just set $\hf=\hg$).

For $\ha,\hb\subset \binom{[n]}{k}$, we say that $\ha,\hb$ are {\it exact cross $t$-intersecting} if $|A\cap B|=t$ for every $A\in \ha$, $B\in \hb$. For $\hf\subset \binom{[n]}{k}$, the $t$-covering number $\tau_t(\hf)$ of $\hf$ is defined as
\[
\tau_t(\hf)=\min\left\{|T|\colon |T\cap F|\geq t \mbox{ for all }F\in \hf\right\}.
\]
For the proofs, we need a result concerning exact cross $t$-intersecting $(t+1)$-uniform families.

\begin{prop}\label{prop-1.4}
Let $n> k> t\geq 2$ and let $\ha,\hb\subset \binom{[n]}{t+1}$ be non-empty exact cross $t$-intersecting. If both $\ha$ and $\hb$ do not contain a sunflower of $k-t+2$ petals with center of size $t$, then one of the following holds:
\begin{itemize}
  \item[(i)] either $|\ha|\leq 2$, $|\hb|\leq k+1$, $\tau_t(\hb)\geq  t+1$ or $|\hb|\leq 2$, $|\ha|\leq k+1$, $\tau_t(\ha)\geq t+1$.
  \item[(ii)]  $\ha\cup \hb$ is  a sunflower with center of size $t$.
  \item[(iii)]   $|\ha||\hb|\leq \frac{(t+2)^2}{2}$.
\end{itemize}
\end{prop}

Let us present some inequalities and notations  needed in the proofs.

\begin{prop}
Let $n,k,i$ be positive integers. Then
\begin{align}\label{ineq-key}
\binom{n-i}{k} \geq \frac{n-ik}{n}\binom{n}{k}, \mbox{\rm \ for } n>ik.
\end{align}
\end{prop}

\begin{proof}
It is easy to check for all $b>a>0$ that
\begin{align}\label{ineq-prekey}
ba>(b+1)(a-1) \mbox{ holds.}
\end{align}
Note that
\[
\frac{\binom{n-i}{k}}{\binom{n}{k}} = \frac{(n-k)(n-k-1)\ldots(n-k-(i-1))}{n(n-1)\ldots(n-(i-1))}.
\]
Applying \eqref{ineq-prekey} repeatedly we see that the numerator is greater  than $(n-1)(n-2)\ldots(n-(i-1))(n-ki)$ implying
\[
\frac{\binom{n-i}{k}}{\binom{n}{k}} \geq 1-\frac{ik}{n}.
\]
Thus \eqref{ineq-key} holds.
\end{proof}

By \eqref{ineq-key}, we obtain that for $c>1$ and $n\geq c(k-t)^2+(t+1)$
\begin{align}\label{ineq-key2}
\binom{n-t-1}{k-t-1} &\leq \frac{n-t-1}{n-t-1-(k-t)(k-t-1)}\binom{n-k-1}{k-t-1}\nonumber\\[5pt]
&\leq \frac{n-t-1}{n-t-1-(k-t)^2}\binom{n-k-1}{k-t-1}\nonumber\\[5pt]
&\leq \frac{c}{c-1}\binom{n-k-1}{k-t-1}.
\end{align}
Moreover, for $c> 2$ we have
\begin{align}\label{ineq-key3}
\binom{n-t-1}{k-t-1}^2 &\leq \left(\frac{c}{c-1}\right)^2\binom{n-k-1}{k-t-1}^2\nonumber\\[5pt]
&= \frac{c^2}{c^2-2c+1}\binom{n-k-1}{k-t-1}^2\nonumber\\[5pt]
&\leq \frac{c}{c-2}\binom{n-k-1}{k-t-1}^2.
\end{align}
Similarly, we can show that for $n\geq ck$ and $c>2$,
\begin{align}\label{ineq-key4}
\binom{n-t-1}{k-t-1}^2 \leq \frac{c}{c-2}\binom{n-t-2}{k-t-1}^2.
\end{align}

%
Let us recall the following common notations:
$$\hf(i)=\{F\setminus\{i\}\colon i\in F\in \hf\}, \ \hf(\bar{i})= \{F\in\hf: i\notin F\}.$$
Note that $|\hf|=|\hf(i)|+|\hf(\bar{i})|$. For $P\subset Q\subset [n]$, let
\[
\hf(Q) =\left\{F\setminus Q\colon Q\subset F\in\hf\right\},\ \hf(P,Q) = \left\{F\setminus Q\colon F\cap Q=P, F\in\hf\right\}.
\]
We also use  $\hf(\bar{Q})$ to denote $\hf(\emptyset, Q)$. For $\hf(\{i\},Q)$ we simply write  $\hf(i,Q)$.

Define the family of {\it $\ell$-th $t$-transversals} of $\hf\subset\binom{[n]}{k}$:
\[
\hht_t^{(\ell)}(\hf) = \left\{T\subset [n]\colon |T|=\ell, |T\cap F|\geq t  \mbox{ for all } F\in \hf\right\}
\]
and define the family of {\it $t$-transversals} of $\hf$:
\[
\hht_t(\hf)=\bigcup_{t\leq \ell\leq k}\hht_t^{(\ell)}(\hf).
\]
Clearly,  if $\hf,\hg$ are cross $t$-intersecting then $\hf\subset \hht_t^{(k)}(\hg)$ and $\hg\subset \hht_t^{(k)}(\hf)$.

The rest of the paper is organized as follows. In Section 2, we recall some inequalities concerning cross-intersecting families that are needed in the proofs. In Section 3, we determine the maximum product size of non-trivial cross $t$-intersecting families with a common $t$-transversal of size $t+1$. In Section 4, we define a notion of basis for cross $t$-intersecting families and establish an upper bound on the size of the basis. In  Section 5, we prove Theorem \ref{thm-main}.

\section{Some inequalities concerning cross-intersecting families}

In this section, we recall several useful inequalities concerning cross-intersecting families. We also give a proof of a result of Hilton via the Hilton-Milner-Frankl Theorem.

An important tool for proving the results concerning cross-intersecting families is the Kruskal-Katona Theorem (\cite{Kruskal,Katona}, cf. \cite{F84} or \cite{Keevash} for short proofs of it).

Daykin \cite{daykin} was the first to show that the Kruskal-Katona Theorem implies the $t=1$ case of the Erd\H{o}s-Ko-Rado Theorem. Hilton \cite{Hilton} gave a very useful reformulation of the Kruskal-Katona Theorem. To state it let us recall the definition of the lexicographic order on $\binom{[n]}{k}$. For two distinct sets $F,G\in \binom{[n]}{k}$ we say that $F$ {\it precedes} $G$ if
\[
\min\{i\colon i\in F\setminus G\}<\min\{i\colon i\in G\setminus F\}.
\]
E.g., $\{1,7\}$ precedes $\{2,3\}$. For a positive integer $b$, let $\hl(n,b,m)$ denote the
first $m$ members of $\binom{[n]}{b}$ in lexicographic order.

\vspace{6pt}
{\noindent\bf Hilton's Lemma (\cite{Hilton}).} Let $n,a,b$ be positive integers, $n>a+b$. Suppose that $\ha\subset \binom{[n]}{a}$ and $\hb\subset \binom{[n]}{b}$ are cross-intersecting. Then $\hl(n,a,|\ha|)$ and $\hl(n,b,|\hb|)$ are cross-intersecting as well.
\vspace{6pt}

One can deduce a result of Hilton \cite{Hilton77} from  the Hilton-Milner-Frankl Theorem.

\begin{thm}[\cite{Hilton77}]\label{Hilton77}
Let $\hk_1,\hk_2,\ldots,\hk_{t+1} \subset \binom{X}{a}$ be pairwise cross-intersecting and $|X|=m\geq (t+1)a$. If at least two of them are non-empty, then
\begin{align}\label{Hilton}
\sum_{i=1}^{t+1}|\hk_i|\leq \max\left\{(t+1)\binom{m-1}{a-1},\binom{m}{a}-\binom{m-a}{a}+t\right\}.
\end{align}
\end{thm}

\begin{proof}
Set $n=m+t+1$, $k=a+t$, $X=[t+2,n]$. Define the family $\hk=\tilde{\hk}_0\cup \ldots\cup \tilde{\hk}_{t+1}\subset \binom{[n]}{k}$ via
\begin{align*}
\tilde{\hk}_0 &=\left\{[t+1]\cup K\colon K\in \binom{X}{k-t-1}\right\},\\[5pt]
\tilde{\hk}_i &=\left\{([t+1]\setminus \{i\})\cup K\colon K\in \hk_i\right\},\ 1\leq i\leq t+1.
\end{align*}
It is easy to check that $\hk$ is $t$-intersecting. Let $I=\cap \{K\colon K\in \hk\}$. Let us show $|I|<t$. Without loss of generality assume that $\hk_1,\hk_2$ are non-empty. It follows that $1,2\notin I$. For any $t+2\leq x\leq n$, $x\notin \cap \{K\colon K\in \tilde{\hk}_0\}\supset I$. Thus $\hk$ is non-trivial. Applying the Hilton-Milner-Frankl Theorem, for $(m+t+1)\geq (a+1)(t+1)$ we have
\begin{align*}
|\hk| &=\binom{n-t-1}{k-t-1} +\sum_{1\leq i\leq t+1}|\tilde{\hk}_i| \\[5pt]
&\leq \max\{|\ha(n,k,t)|,|\hh(n,k,t)|\}\\[5pt]
&=\max\left\{(t+2)\binom{n-t-2}{k-t-1}+\binom{n-t-2}{k-t-2}, t+\binom{n-t}{k-t}-\binom{n-k-1}{k-t}\right\}\\[5pt]
&=\binom{n-t-1}{k-t-1}+ \max\left\{(t+1)\binom{m-1}{a-1},\binom{m}{a}-\binom{m-a}{a}+t\right\},
\end{align*}
implying \eqref{Hilton}.
\end{proof}

{\noindent \bf Remark.} It should be mentioned that if $\hf$ is non-trivial $t$-intersecting
and has a $t$-transversal of size $t+1$, then one can also deduce the Hilton-Milner-Frankl Theorem
from Theorem \ref{Hilton77}.

\begin{cor}
Let $\hk_1,\hk_2\subset \binom{X}{a}$ be non-empty  cross-intersecting and $|X|=m\geq (t+1)a$. If $|\hk_1|\geq |\hk_2|$, then
\begin{align}\label{Hilton2}
|\hk_1|+t|\hk_2|\leq \max\left\{(t+1)\binom{m-1}{a-1},\binom{m}{a}-\binom{m-a}{a}+t\right\}.
\end{align}
\end{cor}

\begin{proof}
By Hilton's Lemma, we may assume that $\hk_i$  consists of the first $|\hk_i|$ members of $\binom{X}{a}$ in lexicographic order. Clearly $|\hk_2|\leq \binom{n-1}{a-1}$, implying that $\hk_2$ is intersecting. Then $\hk_1,\hk_2,\ldots,\hk_2$ are all non-empty
 and pairwise cross-intersecting. By  \eqref{Hilton} the corollary follows.
\end{proof}

We should mention that Hilton's result holds for the full range $m\geq 2a$, but we only use it in the range that we proved it.

\begin{thm}[\cite{F87}]
Let $\hk_1,\hk_2\subset \binom{X}{a}$ be non-empty cross-intersecting and $|X|=m\geq 2a$. If $|\hk_1|\geq |\hk_2|\geq  \binom{m-2}{a-2}$, then
\begin{align}\label{FT92}
|\hk_1|+|\hk_2|\leq 2\binom{m-1}{a-1}.
\end{align}
\end{thm}

\section{Non-trivial cross $t$-intersecting families with common trasversals}

In this section, we determine the maximum product sizes of two non-trivial cross $t$-intersecting families with a common $t$-transversal of size $t+1$.

\begin{prop}
Suppose that $\hf,\hg\subset \binom{[n]}{k}$ are non-trivial cross $t$-intersecting
and
\[
\hht_t^{(t+1)}(\hf)\cap \hht_t^{(t+1)}(\hg)\neq \emptyset.
\]
If $n\geq 6(t+1)k^2$, then
\begin{align}\label{nonemptycommon}
|\hf||\hg|\leq \max\left\{|\ha(n,k,t)|^2,|\hh(n,k,t)|^2\right\}.
\end{align}
\end{prop}

\begin{proof}
Without loss of generality, assume that $[t+1]\in \hht_t^{(t+1)}(\hf)\cap \hht_t^{(t+1)}(\hg)$. Since $[t+1]$ is a common $t$-transversal, for $\hh=\hf$ or $\hg$
\begin{align}\label{ineq-total0}
|\hh|=\sum_{1\leq i\leq t+1} |\hh([t+1]\setminus\{i\},[t+1])|+|\hh([t+1])|.
\end{align}

 Assume that $\hf$ and $\hg$ form a saturated pair. Then
 \[
 \hh([t+1]) =\left\{H\setminus [t+1]\colon [t+1]\subset H \in \binom{[n]}{k}\right\}.
 \]
 I.e., the last term in \eqref{ineq-total0} is $\binom{n-t-1}{k-t-1}$.

 For $1\leq i\leq t+1$ set $\hh_i=\hh([t+1]\setminus \{i\},[t+1])$. For convenience let $X =[t+2,n]$. Then $\hh_i\subset \binom{X}{k-t}$. Note that \eqref{ineq-total0} can be rewritten as
 \begin{align}\label{ineq-total}
 |\hh|=\binom{n-t-1}{k-t-1} +\sum_{1\leq i\leq t+1}|\hh_i|.
 \end{align}
 Note  also that $\hf_i,\hg_j$ are cross-intersecting for $i\neq j$.

 \begin{fact}\label{fact3}
 If $\hf_i,\hg_j$ are cross-intersecting and $\hg_j$ is non-empty, then
 \[
 |\hf_i|\leq (k-t)\binom{n-t-1}{k-t-1}.
 \]
 \end{fact}
 \begin{proof}
  Let $G\in \hg_j$. Since each member in $\hf_i$ intersects $G$ and $|G| = k-t$, we infer
 \[
 |\hf_i|  \leq \binom{|X|}{k-t}-\binom{|X|-|G|}{k-t}\leq (k-t)\binom{|X|-1}{k-t-1}
 \leq(k-t)\binom{n-t-1}{k-t-1}.
 \]
 \end{proof}

 By Hilton's Lemma, we may assume that $\hh_i$ consists of the first $|\hh_i|$ members of $\binom{X}{k-t}$ in lexicographic order. By symmetry assume that
 \[
 |\hf_1|\geq |\hf_2|\geq \ldots\geq |\hf_{t+1}| \mbox{ and thereby } \hf_1\supset \hf_2\supset \ldots\supset \hf_{t+1}.
 \]
 Let $\hg^*$ be the one of $\hg_1,\hg_2,\ldots,\hg_{t+1}$ with the maximum size.
\begin{fact}\label{fact4}
 We may assume that $\min\{|\hf_1|,|\hg^*|\}\geq \binom{n-t-3}{k-t-2}$.
\end{fact}
 \begin{proof}
 By non-triviality of $\hf$, $\hg$, we see that both $\hf_1$ and $\hg^*$ are non-empty. Then by Fact \ref{fact3}
 \[
 |\hf_1|\leq (k-t)\binom{n-t-1}{k-t-1}\mbox{ and }|\hg^*|\leq (k-t)\binom{n-t-1}{k-t-1}.
 \]

If $\min\{|\hf_1|,|\hg^*|\}< \binom{n-t-3}{k-t-2}$, without loss of generality assume that $|\hg^*|< \binom{n-t-3}{k-t-2}$. Then
  for $n\geq 4(t+1)^2(k-t+1)+t+1$ we obtain that
 \begin{align*}
 |\hf||\hg|\leq &\left(\binom{n-t-1}{k-t-1}+(t+1)|\hf_1|\right) \left(\binom{n-t-1}{k-t-1}+(t+1)|\hg^*|\right)\\[5pt]
 \leq &\left(\binom{n-t-1}{k-t-1}+(t+1)(k-t)\binom{n-t-1}{k-t-1}\right) \left(\binom{n-t-1}{k-t-1}+(t+1)\binom{n-t-3}{k-t-2}\right)\\[5pt]
\leq & ((t+1)(k-t)+1)\binom{n-t-1}{k-t-1}^2+ ((t+1)(k-t)+1)(t+1)\binom{n-t-2}{k-t-2}\binom{n-t-1}{k-t-1}\\[5pt]
\leq & (t+1)(k-t+1)\binom{n-t-1}{k-t-1}^2+\frac{(t+1)^2(k-t+1)^2}{n-t-1}\binom{n-t-1}{k-t-1}^2\\[5pt]
\leq & \left(t+\frac{5}{4}\right)(k-t+1)\binom{n-t-1}{k-t-1}^2.
 \end{align*}
Then apply  \eqref{ineq-key3} with $c=4(t+1)$ and note that $\left(t+\frac{5}{4}\right)(t+1)<\left(t+\frac{1}{2}\right)(t+2)$ for $t\geq 2$, we obtain that
  \begin{align*}
 |\hf||\hg| &\leq  \left(t+\frac{5}{4}\right)\frac{4t+4}{4t+2}(k-t+1)\binom{n-k-1}{k-t-1}^2\nonumber\\[5pt]
 &\leq (t+2)(k-t+1)\binom{n-k-1}{k-t-1}^2\nonumber\\[5pt]
 &\leq \max\left\{|\ha(n,k,t)|^2,|\hh(n,k,t)|^2\right\}.
 \end{align*}
 Thus we may assume that $\min\{|\hf_1|,|\hg^*|\}\geq \binom{n-t-3}{k-t-2}$.
 \end{proof}

 Now we distinguish two cases.

 \vspace{6pt}
 {\noindent\bf Case 1.} $|\hg_1|\geq \max\limits_{2\leq i\leq t+1} |\hg_i|$.
 \vspace{6pt}

  Without loss of generality, assume that $|\hg_2|=\max\limits_{2\leq i\leq t+1} |\hg_i|$.
  Then $\hg_j\subset\hg_2$ for $3\leq j\leq t+1$. Since $\hf$ is non-trivial,
  we know that $\hf_2\neq \emptyset$. For otherwise $\hf_2=\hf_3=\ldots=\hf_{t+1}=\emptyset$,
  it follows that $\hf\subset \{F\in \binom{[n]}{k}\colon [2,t+1]\subset F\}$, contradicting the non-triviality of $\hf$. Similarly $\hg_2\neq \emptyset$.

  \vspace{6pt}
  {\noindent\bf Subcase 1.1.} $|\hf_1|\geq |\hg_2|$ and $|\hg_1|\geq |\hf_2|$.
  \vspace{6pt}

  Since $\hf_1,\hg_2$ is cross-intersecting and $|\hf_1|\geq |\hg_2|$, by \eqref{Hilton2} we obtain that
  \begin{align}\label{ineq-3.4}
  |\hf_1|+t|\hg_2| \leq \max\left\{(t+1)\binom{|X|-1}{k-t-1},\binom{|X|}{k-t}-\binom{|X|-(k-t)}{k-t}+t\right\}.
  \end{align}
  Similarly,
    \begin{align}\label{ineq-3.5}
  |\hg_1|+t|\hf_2| \leq \max\left\{(t+1)\binom{|X|-1}{k-t-1},\binom{|X|}{k-t}-\binom{|X|-(k-t)}{k-t}+t\right\}.
  \end{align}
 By \eqref{ineq-total}, \eqref{ineq-3.4} and \eqref{ineq-3.5}, we arrive at
\begin{align*}
|\hf|+|\hg|&=  2\binom{n-t-1}{k-t-1} +\sum_{1\leq i\leq t+1}|\hf_i|+\sum_{1\leq i\leq t+1}|\hg_i|\\[5pt] &\leq 2\binom{n-t-1}{k-t-1}+2\max\left\{(t+1)\binom{|X|-1}{k-t-1},\binom{|X|}{k-t}-\binom{|X|-(k-t)}{k-t}+t\right\}\\[5pt]
&= 2\max\left\{(t+2)\binom{n-t-2}{k-t-1}+\binom{n-t-2}{k-t-2}, t+\binom{n-t}{k-t}-\binom{n-k-1}{k-t}\right\}
\end{align*}
and \eqref{nonemptycommon} follows from $|\hf||\hg| \leq \left(\frac{|\hf|+|\hg|}{2}\right)^2$.

 \vspace{6pt}
 {\noindent\bf Subcase 1.2.} $|\hg_1|< |\hf_2|$ or $|\hf_1|< |\hg_2|$.
 \vspace{6pt}

 By  symmetry, we may assume that $|\hg_1|< |\hf_2|$. By Fact \ref{fact4} we see $|\hg_1|\geq  \binom{n-t-3}{k-t-2}$.
  Since $\hg_1,\hf_2$ are cross-intersecting, by \eqref{FT92} we have
 \begin{align}\label{ineq-part1}
 |\hg_1|+|\hf_2|\leq  2\binom{|X|-1}{k-t-1},
 \end{align}
 and thereby $|\hg_1|\leq \binom{n-t-2}{k-t-1}$.

  \vspace{6pt}
 {\noindent\bf Subcase 1.2.1. } $|\hg_2|\geq  \binom{n-t-3}{k-t-2}$.
 \vspace{6pt}

Since $\hg_2,\hf_1$ are cross-intersecting and $|\hf_1|\geq |\hg_2|\geq \binom{n-t-3}{k-t-2}$, by \eqref{FT92}
 \begin{align}\label{ineq-part2}
 |\hf_1|+|\hg_2| \leq 2\binom{|X|-1}{k-t-1}.
 \end{align}
By \eqref{ineq-part1} and \eqref{ineq-part2}, we obtain that
\begin{align*}
|\hf|+|\hg|&\leq 2\binom{n-t-1}{k-t-1}+ t(|\hg_1|+|\hf_2|) +( |\hf_1|+|\hg_2|)\\[5pt]
 &\leq 2\binom{n-t-1}{k-t-1}+ (2t+2)\binom{n-t-2}{k-t-1}\\[5pt]
 &= 2(t+2)\binom{n-t-2}{k-t-1}+2\binom{n-t-2}{k-t-2}\\[5pt]
  &= 2|\ha(n,k,t)|
\end{align*}
and \eqref{nonemptycommon} follows.

  \vspace{6pt}
 {\noindent\bf Subcase 1.2.2. } $|\hg_2|< \binom{n-t-3}{k-t-2}< \binom{n-t-2}{k-t-2}$.
 \vspace{6pt}

 By Fact \ref{fact3} we have $ |\hf_1| \leq (k-t)\binom{n-t-1}{k-t-1}$.   Since  $\hf_2,\hg_1$ are cross-intersecting, by \eqref{ineq-pyber} we obtain that
 \begin{align}\label{ineq-3.8}
 |\hf_2||\hg_1|\leq \binom{n-t-2}{k-t-1}^2\leq \binom{n-t-1}{k-t-1}^2.
 \end{align}
  Therefore,
 \begin{align*}
 |\hf||\hg|\leq &\left(\binom{n-t-1}{k-t-1}+t|\hf_2|+|\hf_1|\right) \left(\binom{n-t-1}{k-t-1}+t|\hg_2|+|\hg_1|\right)\nonumber\\[5pt]
\leq &\left((k-t+1)\binom{n-t-1}{k-t-1}+t|\hf_2|\right) \left(\binom{n-t-1}{k-t-1}+t\binom{n-t-2}{k-t-2}+|\hg_1|\right)\nonumber\\[5pt]
= &\left((k-t+1)\binom{n-t-1}{k-t-1}+t|\hf_2|\right) \left(\left(1+\frac{t(k-t-1)}{n-t-1}\right)\binom{n-t-1}{k-t-1}+|\hg_1|\right)\nonumber\\[5pt]
= &(k-t+1)\left(1+\frac{t(k-t-1)}{n-t-1}\right)\binom{n-t-1}{k-t-1}^2+t|\hf_2||\hg_1|\nonumber\\[5pt]
&\quad +\binom{n-t-1}{k-t-1}\left(\left(1 +\frac{t(k-t-1)}{n-t-1}\right)t|\hf_2|+(k-t+1)|\hg_1|\right).
\end{align*}
By \eqref{ineq-3.8}, it follows that
 \begin{align}\label{ineq-subcase1222}
 |\hf||\hg|\leq &(k+1)\binom{n-t-1}{k-t-1}^2+\frac{t(k-t)^2}{n-t-1}\binom{n-t-1}{k-t-1}^2\nonumber\\[5pt]
&\quad +\binom{n-t-1}{k-t-1}\left(\left(1 +\frac{t(k-t-1)}{n-t-1}\right)t|\hf_2|+(k-t+1)|\hg_1|\right).
 \end{align}
By \eqref{ineq-part1} and $|\hg_1|\leq \binom{n-t-1}{k-t-1}$, we have
\begin{align}\label{ineq-subcase1221}
&\left(1+\frac{t(k-t-1)}{n-t-1}\right)t|\hf_2|+(k-t+1)|\hg_1|\nonumber\\[5pt]
< & \left(1+\frac{t(k-t)}{n-t-1}\right)t\left(2\binom{n-t-1}{k-t-1}-|\hg_1|\right)+(k-t+1)|\hg_1|\nonumber\\[5pt]
\leq &\max\left\{2t+\frac{2t^2(k-t)}{n-t-1}, k+1+\frac{t^2(k-t)}{n-t-1}\right\}\binom{n-t-1}{k-t-1}.
\end{align}
Combining \eqref{ineq-subcase1221} and \eqref{ineq-subcase1222}, we arrive at
\begin{align*}\label{ineq-subcase1223}
 |\hf||\hg|
\leq &\left(\frac{t(k-t)^2}{n-t-1}+\max\left\{k+1+2t+\frac{2t^2(k-t)}{n-t-1}, 2(k+1)+\frac{t^2(k-t)}{n-t-1}\right\}\right)\\[5pt]
&\qquad\cdot\binom{n-t-1}{k-t-1}^2.
\end{align*}
Note that
\[
t(k-t)^2+2t^2(k-t)=t(k-t)(k+t)=t(k^2-t^2)\leq tk^2.
\]
For $n\geq tk^2+(t+1)$,
\begin{align}\label{ineq-subcase1223}
 |\hf||\hg|< & \binom{n-t-1}{k-t-1}^2\max\left\{k+2t+1+\frac{tk^2}{n-t-1}, 2(k+1)+\frac{tk^2}{n-t-1}\right\}\nonumber\\[5pt]
\leq&  \binom{n-t-1}{k-t-1}^2\max\left\{k+2t+2, 2k+3\right\}.
 \end{align}

 If $k\geq 2t+3$, then $k+2t+2< 2k+3$ and $2k+3= 4(k-\frac{k}{2}+\frac{3}{4})< 4(k-t+1)$.
 By applying  \eqref{ineq-key3} with $c=6$,  we obtain for $n\geq 6(k-t)^2+t+1$,
 \[
  |\hf||\hg|\leq (2k+3)\binom{n-t-1}{k-t-1}^2\leq 4(k-t+1)\binom{n-t-1}{k-t-1}^2\leq 6(k-t+1)\binom{n-k-1}{k-t-1}^2.
 \]
Since $k\geq 2t+3$ implies $k-t+1\geq t+4\geq 6$, we have
\[
 |\hf||\hg|\leq  6(k-t+1)\binom{n-k-1}{k-t-1}^2\leq (k-t+1)^2\binom{n-k-1}{k-t-1}^2<|\hh(n,k,t)|^2.
\]
If $2t\leq k\leq 2t+2$, then $k+2t+2\leq  2k+3$. By \eqref{ineq-subcase1223} we have
 \[
  |\hf||\hg|\leq (2k+3)\binom{n-t-1}{k-t-1}^2\leq (4t+7)\binom{n-t-1}{k-t-1}^2.
 \]
 Applying \eqref{ineq-key4} with $c=8(t+2)$, we arrive at
 \[
  |\hf||\hg|\leq (4t+7)\frac{8(t+2)}{8(t+2)-2}\binom{n-t-2}{k-t-1}^2= 4(t+2)\binom{n-t-2}{k-t-1}^2<|\ha(n,k,t)|^2.
 \]
 If $k\leq 2t-1$, then $k+2t+2\geq  2k+3$. By \eqref{ineq-subcase1223} we have
 \[
  |\hf||\hg|\leq (k+2t+2)\binom{n-t-1}{k-t-1}^2< (4t+4)\binom{n-t-1}{k-t-1}^2.
 \]
 Applying \eqref{ineq-key4} with $c=2(t+2)$, we arrive at
 \[
  |\hf||\hg|\leq 4(t+1)\frac{2(t+2)}{2(t+2)-2}\binom{n-t-2}{k-t-1}^2= 4(t+2)\binom{n-t-2}{k-t-1}^2<|\ha(n,k,t)|^2.
 \]

\vspace{6pt}
 {\noindent\bf Case 2.}  $|\hg_1|\leq \max\limits_{2\leq i\leq t+1} |\hg_i|$.
 \vspace{6pt}

 Without loss of generality, assume that $|\hg_2|=\max\limits_{2\leq i\leq t+1} |\hg_i|$. Then $|\hg_2|=\max\limits_{1\leq i\leq t+1} |\hg_i|$ and $\hg_i\subset \hg_2$ for all $1\leq i\leq t+1$. Since $\hf_i\subset \hf_1$, $\hg_i\subset \hg_2$, $\hf_i$ and $\hg_i$ are cross-intersecting for $1\leq i\leq t+1$, by saturatedness we may assume that $|\hf_i|= |\hf_1|$ and $|\hg_i|= |\hg_2|$ for all $i\in [t+1]$. Then \eqref{nonemptycommon} is equivalent to
 \begin{align}\label{ineq-3}
 \left(\binom{n-t-1}{k-t-1}+(t+1)|\hf_1|\right) \left(\binom{n-t-1}{k-t-1}+(t+1)|\hg_2|\right)\leq \max\{|\ha(n,k,t)|^2,|\hh(n,k,t)|^2\}.
 \end{align}

 Without loss of generality, assume that $|\hf_1|\leq |\hg_2|$.
 By Fact \ref{fact4} we have $|\hf_1|\geq \binom{n-t-3}{k-t-2}$.
 Since $\hf_1,\hg_2$ are cross-intersecting, by \eqref{FT92} we have
 $|\hf_1|+|\hg_2|\leq 2\binom{|X|-1}{k-t-1}$.
By \eqref{ineq-pyber} $|\hf_1||\hg_2|\leq \binom{|X|-1}{k-t-1}^2$. Then expanding \eqref{ineq-3}:
 \begin{align*}
 LHS&\leq \binom{n-t-1}{k-t-1}^2+2\binom{n-t-1}{k-t-1}(t+1)\binom{n-t-2}{k-t-1}
 +\left((t+1)\binom{n-t-2}{k-t-1}\right)^2\\[5pt]
 &=|\ha(n,k,t)|^2,
 \end{align*}
and the proposition is proven.
\end{proof}

\section{The basis of cross $t$-intersecting families}

In this section, we prove an inequality concerning the size of  basis of cross $t$-intersecting families by a branching process.

We need the following  notion of  basis.  Let $\hb(\hf)$ be the family of minimal (for containment) sets in $\hht_t(\hg)$ and let $\hb(\hg)$ be the family of minimal sets in
$\hht_t(\hf)$.

Let us prove some properties of the basis.

\begin{lem}\label{lem3-1}
Suppose that $\hf,\hg\subset \binom{[n]}{k}$ form a saturated pair of cross $t$-intersecting families. Then (i) and (ii) hold.
\begin{itemize}
  \item[(i)] Both $\hb(\hf)$ and $\hb(\hg)$ are antichains, and $\hb(\hf), \hb(\hg)$ are cross $t$-intersecting,
  \item[(ii)] $\hf=\left\{F\in \binom{[n]}{k}\colon \exists B\in \hb(\hf), B\subset F\right\}$ and $\hg=\left\{G\in \binom{[n]}{k}\colon \exists B\in \hb(\hg), B\subset G\right\}$.
\end{itemize}
\end{lem}

\begin{proof}
 (i) Clearly,  $\hb(\hf)$ and $\hb(\hg)$ are both antichains. Suppose for contradiction
 that $B\in \hb(\hf), B'\in \hb(\hg)$ but $|B\cap B'|<t$. If $|B|=|B'|=k$, then
 $B\in \hf$, $B'\in \hg$ follow from saturatedness, a contradiction. If $|B|<k$,
 then there exists $F\supset B$ such that $|F|=k$ and $|F\cap B'|=|B\cap B'|<t$.
 By definition $F\in \hht_t(\hg)$. Since $\hf,\hg$ are saturated, we see that $F\in \hf$.
 But this contradicts the assumption that $B'$ is a $t$-transversal of $\hf$.
 Since $\hf,\hg$ are saturated, (ii) is immediate from the definition of $\hb(\hf)$ and  $\hb(\hg)$.
\end{proof}

Let $s(\hb)=\min\{|B|\colon B\in \hb\}$. For any $\ell$ with $s(\hb)\leq \ell \leq k$, define
\[
\hb^{(\ell)} = \left\{B\in \hb\colon |B|=\ell\right\} \mbox{ and } \hb^{(\leq \ell)} = \bigcup_{i=s(\hb)}^\ell\hb^{(i)}.
\]
It is easy to see that $s(\hb(\hg))=\tau_t(\hf)$.

By  a branching process, we establish an inequality concerning the size of the basis.

\begin{lem}\label{lem3-2}
Suppose that $\hf,\hg\subset \binom{[n]}{k}$ are a saturated pair of cross $t$-intersecting families.
Let $\hb_1=\hb(\hf)$ and $\hb_2=\hb(\hg)$. For each $i=1,2$, if $s(\hb_i)\geq t+1$ and there exists $r_i\geq s(\hb_i)$ such that
$\tau_t(\hb_i^{(\leq r_i)})\geq t+1$  then
\begin{align}\label{ineq-crosshb}
\sum_{r_i\leq \ell\leq k}\left(\binom{s(\hb_i)}{t}r_i k^{\ell-t-1}\right)^{-1}|\hb_{3-i}^{(\ell)}|\leq 1
\end{align}
and
\begin{align}\label{ineq-crosshb2}
\sum_{s_i\leq \ell\leq k}\left(\binom{s(\hb_i)}{t}k^{\ell-t}\right)^{-1}|\hb_{3-i}^{(\ell)}|\leq 1.
\end{align}
\end{lem}

\begin{proof}
By symmetry, it is sufficient to prove the lemma only for $i=1$.
For the proof we use a branching process. During the proof {\it a sequence}
$S=(x_1,x_2,\ldots,x_\ell)$ is an ordered sequence of distinct elements of $[n]$
and we use $\widehat{S}$ to denote the underlying unordered set $\{x_1,x_2,\ldots,x_\ell\}$.
At the beginning, we assign weight 1 to the empty sequence $S_{\emptyset}$. At the first stage,
we choose $B_{1,1}\in \hb_1$ with $|B_{1,1}|=s(\hb_1)\geq t+1$. For any $t$-subset
$\{x_1,x_2,\ldots,x_t\}\subset B_{1,1}$, define one sequence $(x_1,x_2,\ldots,x_t)$ and
assign the weight $\binom{s(\hb_1)}{t}^{-1}$ to it.

At the second stage, since $\tau_t(\hb_1^{(\leq r_1)})\geq t+1$,
for each sequence $S=(x_1,\ldots,x_t)$ we may choose $B_{1,t+1}\in \hb_1^{(\leq r_1)}$
such that $|\widehat{S}\cap B_{1,t+1}|<t$. Then we replace $S=(x_1,\ldots,x_t)$ by
$|B_{1,t+1}\setminus \widehat{S}|$ sequences of the form $(x_1,\ldots,x_t,y)$ with
$y\in B_{1,t+1}\setminus \widehat{S}$ and weight $\frac{w(S)}{|B_{1,t+1}\setminus \widehat{S}|}$.

In each subsequent stage, we pick a sequence $S=(x_1,\ldots,x_p)$ and denote its weight by $w(S)$. If $|\widehat{S}\cap B_1|\geq t$ holds for all $B_1\in \hb_1$ then we do nothing. Otherwise we pick $B_1\in \hb_1$ satisfying $|\widehat{S}\cap B_1|<t$ and replace $S$ by the $|B_1\setminus \widehat{S}|$ sequences $(x_1,\ldots,x_p,y)$ with $y\in B_1\setminus \widehat{S}$ and assign weight $\frac{w(S)}{|B_1\setminus \widehat{S}|}$ to each of them. Clearly, the total weight is always 1.

We continue until $|\widehat{S}\cap B_1|\geq t$ for all sequences and all $B_1\in \hb_1$. Since $[n]$ is finite, each sequence has length at most $n$ and  eventually the process stops. Let $\hs$ be the collection of sequences that survived in the end of the branching process and let $\hs^{(\ell)}$ be the collection of sequences in $\hs$ with length $\ell$.

\begin{claim}
To each $B_2\in \hb_{2}^{(\ell)}$ with $\ell\geq r_1$ there is some sequence $S\in \hs^{(\ell)}$ with $\widehat{S}=B_2$.
\end{claim}
\begin{proof}
Let us suppose the contrary and let $S=(x_1,\ldots,x_p)$ be a sequence of maximal length that occurred
at some stage of the branching process satisfying $\widehat{S}\subsetneqq B_2$. Since $\hb_1,\hb_2$ are cross $t$-intersecting,
 $|B_{1,1}\cap B_2|\geq  t$, implying that such an $S$ exists. By the choice of $S$ we see that $p\geq t$. Since $\widehat{S}$ is a
 proper subset of $B_2$ and $B_2\in \hb_2=\hb(\hg)$, it follows that
 $\widehat{S}\notin \hb(\hg)\subset\hht(\hf)$. Thereby there exists $F\in \hf$ with
 $|\widehat{S} \cap F|< t$.  In view of Lemma \ref{lem3-1} (ii), we can find  $B_1'\in \hb_1$
 such that $|\widehat{S} \cap B_1'|<t$. Thus at some point we picked $S$ and
 some $\tilde{B}_1\in \hb_1$ with $|\widehat{S} \cap \tilde{B}_1|<t$. Since $\hb_1,\hb_2$
 are cross $t$-intersecting, $|B_2\cap \tilde{B}_1|\geq t$. Consequently, for each
 $y\in B_2\cap \tilde{B}_1$ the sequence $(x_1,\ldots,x_p,y)$ occurred in the branching process.
 This contradicts the maximality of $p$. Hence there is an $S$ at some stage satisfying
 $\widehat{S}= B_2$. Since $\hb_1,\hb_2$ are cross $t$-intersecting,
 $|\widehat{S}\cap B_1'|=|B_2\cap B_1'|\geq t$ for all $B_1'\in \hb_1$.
 Thus $\widehat{S}\in \hs$ and the claim holds.
\end{proof}
By Claim 1, we see that $|\hb_2^{(\ell)}|\leq |\hs^{(\ell)}|$ for all $\ell\geq r_1$.
Let $S=(x_1,\ldots,x_\ell)\in \hs^{(\ell)}$ and let $S_i=(x_1,\ldots,x_i)$ for $i=t,\ldots,\ell$.
At the first stage, $w(S_t)=1/\binom{s(\hb_1)}{t}$.   Assume that $B_{1,i}$ is the selected set
when replacing $S_{i-1}$ in the  branching process for $i=t+1,\ldots,\ell$. Clearly,
$x_i\in B_{1,i}$ for $i\geq t+1$, $B_{1,t+1}\in \hb_1^{(\leq r_1)}$ and
\[
w(S)= \binom{s(\hb_1)}{t}^{-1}\prod_{i=t+1}^\ell \frac{1}{|B_{1,i}\setminus \widehat{S_{i-1}}|}.
\]
Note that $|B_{1,t+1}\setminus \widehat{S_t}|\leq r_1$ and $|B_{1,i}\setminus \widehat{S_{i-1}}|\leq k$ for $i\geq t+2$. It follows that
\[
w(S)\geq \left(\binom{s(\hb_1)}{t}r_1 k^{\ell-t-1}\right)^{-1}.
\]
Thus we obtain that
\[
\sum_{r_1\leq \ell\leq k}\left(\binom{s(\hb_1)}{t}r_1 k^{\ell-t-1}\right)^{-1}|\hb_{2}^{(\ell)}|\leq \sum_{r_1\leq \ell\leq k}\sum_{S\in \hs^{(\ell)}}w(S)\leq \sum_{S\in \hs}w(S)=1
\]
and \eqref{ineq-crosshb} holds.

If we omit the second stage in the branching process, then by the similar argument we can obtain
 \eqref{ineq-crosshb2}.
\end{proof}

\section{The proof of the main theorem}

In this section, we determine the maximum product of  the sizes of two non-trivial cross $t$-intersecting families.

\begin{lem}\label{lem-5.1}
Let $\hf,\hg\subset \binom{[n]}{k}$ be a saturated pair of non-trivial cross $t$-intersecting families. Set
$\hb_1=\hb(\hf),\hb_2=\hb(\hg)$. Then neither $\hb_1$ nor $\hb_2$ contains a sunflower of  $k-t+2$ petals with center of size $t$.
\end{lem}
\begin{proof}
Suppose for contradiction that there exists a sunflower of $k-t+2$ petals with center of size $t$ in $\hb_1$. Without loss of generality, assume that $[t]\cup A_1,\ldots, [t]\cup A_{k-t+2}$ is such a sunflower. Then for each $G\in \hg$, by the definition of $\hb(\hf)$ we have $|G\cap ([t]\cup A_j)|\geq t$ for $j=1,\ldots,k-t+2$. It follows that $[t]\subset G$, contradicting the non-triviality of $\hg$.
\end{proof}
\begin{lem}
Let $\hf,\hg\subset \binom{[n]}{k}$ be a saturated pair of non-trivial cross $t$-intersecting families. Set
$\ha_1=\hht_t^{(t+1)}(\hg),\ha_2=\hht_t^{(t+1)}(\hf)$. If $s(\hb(\hf))=s$, then
\begin{align}\label{ineq-hbub}
|\ha_2| \leq \binom{s}{t} (k-t+1).
\end{align}
If $\ha_1\neq \emptyset$ and $\ha_1\cap \ha_2= \emptyset$, then
\begin{align}\label{ineq-hbub2}
|\ha_2| \leq 2 (k-t+1).
\end{align}
\end{lem}
\begin{proof}
Let us prove \eqref{ineq-hbub} first. Since $s(\hb(\hf))=s$, there exists an $S\in \hb(\hf)$ with $|S|=s$. Note that  by Lemma \ref{lem3-1} (i) we have $|A\cap S|\geq t$ for each $A\in \ha_2$. Then
\[
|\ha_2| \leq \sum_{T\in \binom{S}{t}} |\ha_2(T)|.
\]
  By Lemma \ref{lem-5.1} we see that  $|\ha_2(T)|\leq k-t+1$ and \eqref{ineq-hbub} follows.

Now we prove \eqref{ineq-hbub2}. Without loss of generality, let $[t+1]\in \ha_1$. Then $|D\cap [t+1]|\geq t$ for all
$D\in \ha_2$. Should $|F\cap [t+1]|\geq t$ hold for all $F\in \hf$, we get $[t+1]\in\ha_1\cap \ha_2$, contradicting $\ha_1\cap \ha_2=\emptyset$.
I.e., we may fix $F_0$ with $|F_0\cap[t+1]|\leq t-1$. Let $\ha_2=\{D_1,\ldots,D_q\}$ and define $x_i$
by  $D_i\setminus [t+1]=\{x_i\}$. Note that
\[
t\leq |F_0\cap D_i|\leq |F_0\cap [t+1]|+|F_0\cap \{x_i\}|\leq t-1+1\leq t.
\]
Thus $|F_0\cap [t+1]|=t-1$, $x_i\in F_0$ and  $F_0\cap [t+1]\subset D_i$. Since $|D_i\cap [t+1]|=t$, it follows that there are only two choices for $D_i\cap [t+1]$. Then by Lemma \ref{lem-5.1} we conclude that $|\ha_2|\leq 2(k-t+1)$.
\end{proof}

\begin{proof}[Proof of Proposition \ref{prop-1.4}]
Let us prove the following two facts.
\begin{fact}\label{fact1}
For any $A,A'\in \ha$, $|A\cap A'|=t-1$ or $t$.
\end{fact}
\begin{proof}
Indeed, if $A\cap A'=D$ with $|D|\leq t-2$, then for any $B\in \hb$ the fact $|B\cap (A'\setminus D)|\geq 2$ implies
\[
|B\cap (A\cup (A'\setminus D))|=|B\cap A| +|B\cap (A'\setminus D)|\geq t+2,
\]
contradicting $|B|=t+1$.
\end{proof}

\begin{fact}\label{fact2}
Suppose that $|A\cap A'|=t-1$ then $A\cap A'\subset B$ for all $B\in \hb$.
\end{fact}
\begin{proof}
Suppose $|B\cap (A\cap A')|\leq t-2$. Then $B\supset A\setminus A'$, $B\supset A'\setminus A$ and $|B\cap A\cap A'|=t-2$ follow. Consequently $|B|=t+2$, contradiction.
\end{proof}

If $|A\cap A'|=t$ for all $A,A'\in \ha$ and also $B\cap B'=t$ for all $B,B'\in \hb$,
then $\ha\cup \hb$ is either a subset of $\binom{[t+2]}{t+1}$ up to isomorphism or
a sunflower with center of size $t$. If $\ha\cup \hb$ is a sunflower with center of size $t$,
then (ii) holds.  If $\ha\cup \hb$ is a subset of $\binom{[t+2]}{t+1}$, note that the exact cross $t$-intersection
 implies   $\ha\cap \hb=\emptyset$, then  $|\ha|+|\hb|\leq t+2$
and thereby $|\ha||\hb|\leq \frac{(t+2)^2}{4}$. Thus (iii) holds.

In the rest of the proof, we assume that there exist $A,A'\in \ha$ such that $|A\cap A'|=t-1$.
By symmetry we assume that $[t+1],[t-1]\cup\{t+2,t+3\}\in \ha$.
By Fact \ref{fact2} $[t-1]\subset B$ for all $B\in \hb$.
We claim that $|\hb|\leq 4$.  Indeed, for every $B\in \hb$,  $|B\cap [t+1]|= t$ and
$|B\cap ([t-1]\cup\{t+2,t+3\})|=t$ imply that $B=[t-1]\cup\{x,y\}$
with $x\in \{t,t+1\}$, $y\in \{t+2,t+3\}$. Now we distinguish two cases.

\vspace{6pt}
{\noindent\bf Case 1. } $\exists B,B'\in \hb$ with $|B\cap B'|=t-1$.
\vspace{6pt}

Necessarily, $B\cap B'=[t-1]$. Consequently  $[t-1]\subset A$ for all $A\in \ha$ and even $A\setminus [t-1]\subset [t,t+3]$. Now $\ha([t-1])$ and $\hb([t-1])$ are cross-intersecting $2$-graphs on the four vertices $t,t+1,t+2,t+3$. We infer $\min\{|\ha|,|\hb|\}=2$. It follows that $|\ha||\hb|\leq 8\leq \frac{(t+2)^2}{2}$ and (iii) holds.

\vspace{6pt}
{\noindent\bf Case 2. } $|B\cap B'|=t$ for every $B,B'\in \hb$.
\vspace{6pt}

Since $B\setminus [t-1]\subset \{t,t+1\}\times \{t+2,t+3\}$, the
only possibility is $|\hb|=2$. Say $\hb=\{[t]\cup\{t+2\},[t]\cup \{t+3\}\}$.

Set $\ha=\ha_0\cup\ha_1$ where $\ha_0=\{A\in \ha\colon [t]\subset A\}$
and $\ha_1=\{A\in \ha\colon |A\cap[t]|=t-1\}$. Now for any $A\in \ha_1$, $|A\cap B|=t$ implies
$\{t+2,t+3\}\subset A$. Hence $|\ha_1|\leq t$. Since $\ha$ does not contain a sunflower of $k-t+1$ petals with center of size $t$, we infer that $|\ha_0|\leq k-t+1$.
If $|\ha_0|\leq 2$, then $|\ha||\hb|\leq 2(t+2)\leq \frac{(t+2)^2}{2}$, (iii) holds.
If $|\ha_0|\geq 3$, since $[t-1]\cup\{t+2,t+3\}\in \ha$, then  $\tau_t(\ha)\geq t+1$. Since  $|\hb|\leq 2$, $|\ha|\leq k+1$, (i) holds.
\end{proof}

\begin{proof}[Proof of Theorem \ref{thm-main}]
Let $\hb_1=\hb(\hf)$, $\hb_2=\hb(\hg)$,  $\ha_1=\hb_1^{(t+1)}$, $\ha_2=\hb_2^{(t+1)}$
and  let $s_1=s(\hb_1)$, $s_2=s(\hb_2)$. Let us partition $\hf$ into $\hf^{(s_1)}\cup \ldots\cup\hf^{(k)}$
 where $F\in \hf^{(\ell)}$ if $\max\{|B|\colon B\in \hb_1, B\subset F\}=\ell$. Similarly,
 partition $\hg$ into $\hg^{(s_2)}\cup \ldots\cup\hg^{(k)}$. By non-triviality of $\hf$ and $\hg$,
 we know $s_1\geq t+1$ and $s_2\geq t+1$. For every $m\in [s_1,k]$, define
 \[
 \hf^{(\geq m)}= \bigcup_{m\leq \ell\leq k} \hf^{(\ell)}.
 \]
 For $m\in [s_2,k]$, define
 \[
  \hg^{(\geq m)}= \bigcup_{m\leq \ell\leq k} \hg^{(\ell)}.
 \]
 Let $\alpha_\ell = \left(\binom{s_2}{t}k^{\ell-t}\right)^{-1}|\hb_1^{(\ell)}|$. By \eqref{ineq-crosshb2} we have
\begin{align}\label{ineq-alhpa}
\sum\limits_{m\leq \ell\leq k}\alpha_\ell\leq \sum\limits_{s_1\leq \ell\leq k}\alpha_\ell \leq 1.
\end{align}
Let $f(n,k,\ell)=\binom{s_2}{t}k^{\ell-t}\binom{n-\ell}{k-\ell}$. For $n\geq k^2$, we have
\begin{align}\label{fnkl}
\frac{f(n,k,\ell+1)}{f(n,k,\ell)}= \frac{k^{\ell+1-t}\binom{n-\ell-1}{k-\ell-1}}{k^{\ell-t}\binom{n-\ell}{k-\ell}}
=\frac{k(k-\ell)}{n-\ell} \leq 1.
\end{align}
Then by \eqref{ineq-alhpa} and \eqref{fnkl} we have for $m\geq s_1$
\begin{align}\label{ineq-case1hf}
|\hf^{(\geq m)}|\leq \sum_{m\leq \ell\leq k} |\hb_1^{(\ell)}|\binom{n-\ell}{k-\ell}
 =&\sum_{m\leq \ell\leq k}\alpha_\ell f(n,k,\ell)\nonumber\\[5pt]
\leq&  f(n,k,m)\nonumber\\[5pt]
 =&\binom{s_2}{t}k^{m-t}\binom{n-m}{k-m}.
\end{align}
Similarly, we have for $m\geq s_2$
\begin{align}\label{ineq-case1hg}
|\hg^{(\geq m)}|\leq \binom{s_1}{t}k^{m-t}\binom{n-m}{k-m}.
\end{align}

 Without loss of generality we assume that $s_1\geq s_2$. Now we distinguish four cases.

\vspace{6pt}
{\noindent\bf Case 1. } $s_1\geq s_2\geq t+2$.
\vspace{6pt}

Applying \eqref{ineq-case1hf} with $m=s_1$, we obtain that
\begin{align}\label{ineq-case1hf2}
|\hf|\leq \binom{s_2}{t}k^{s_1-t}\binom{n-s_1}{k-s_1}.
\end{align}
Applying \eqref{ineq-case1hg} with $m=s_2$, we obtain that
\begin{align}\label{ineq-case1hg2}
|\hg|\leq \binom{s_1}{t}k^{s_2-t}\binom{n-s_2}{k-s_2}.
\end{align}
Then
\begin{align*}
|\hf||\hg|\leq \binom{s_1}{t}k^{s_1-t}\binom{n-s_1}{k-s_1}\binom{s_2}{t}k^{s_2-t}\binom{n-s_2}{k-s_2}.
\end{align*}
Let $g_i(n,k,s_i)= \binom{s_i}{t}k^{s_i-t}\binom{n-s_i}{k-s_i}$, $i=1,2$. Since $s_i\geq t+2$ and for $n\geq tk^2$,
\begin{align}\label{gnksi}
\frac{g_i(n,k,s_i+1)}{g_i(n,k,s_i)} =\frac{\binom{s_i+1}{t}k^{s_i+1-t}\binom{n-s_i-1}{k-s_i-1}}{\binom{s_i}{t}k^{s_i-t}\binom{n-s_i}{k-s_i}}
=\frac{s_i+1}{s_i+1-t}\frac{k(k-s_i)}{n-s_i}\leq \frac{t+3}{3}\frac{k(k-s_i)}{n-s_i}<1,
\end{align}
it follows that for $n\geq (t+2)^2k^2+t+1$,
\begin{align*}
|\hf||\hg|\leq g_1(n,k,s_1)g_2(n,k,s_2)&\leq g_1(n,k,t+2)g_2(n,k,t+2)\\[5pt]
&= \binom{t+2}{t}^2k^{4}\binom{n-t-2}{k-t-2}^2\\[5pt]
&\leq \frac{(t+2)^4k^4 (k-t-1)^2}{4(n-t-1)^2}\binom{n-t-1}{k-t-1}^2 \\[5pt]
&\leq \frac{(k-t+1)^2}{4} \binom{n-t-1}{k-t-1}^2.
\end{align*}
Apply \eqref{ineq-key3} with $c=4$, we conclude that
\[
|\hf||\hg| \leq \frac{(k-t+1)^2}{2} \binom{n-k-1}{k-t-1}^2< |\hh(n,k,t)|^2.
\]

\vspace{6pt}
{\noindent\bf Case 2. } $s_1\geq t+3$ and $s_2=t+1$.
\vspace{6pt}

By  \eqref{ineq-hbub} we have $|\ha_2| \leq \binom{s_1}{t}(k-t+1)$.  By \eqref{ineq-case1hf},
\[
|\hf|\leq (t+1)k^{s_1-t}\binom{n-s_1}{k-s_1}.
\]
 By \eqref{ineq-case1hg},
\[
|\hg^{(\geq t+2)}|\leq \binom{s_1}{t} k^2\binom{n-t-2}{k-t-2}.
\]
It follows that
\begin{align*}
|\hf||\hg|&\leq (t+1)k^{s_1-t}\binom{n-s_1}{k-s_1}
\left(|\ha_2|\binom{n-t-1}{k-t-1}+\binom{s_1}{t}k^2\binom{n-t-2}{k-t-2}\right)\\[5pt]
&\leq (t+1)k^{s_1-t}\binom{n-s_1}{k-s_1}
\left(\binom{s_1}{t}(k-t+1)\binom{n-t-1}{k-t-1}+\binom{s_1}{t}k^2\binom{n-t-2}{k-t-2}\right)\\[5pt]
&\leq (t+1)g_1(n,k,s_1)
\left((k-t+1)\binom{n-t-1}{k-t-1}+k^2\binom{n-t-2}{k-t-2}\right).
\end{align*}
By \eqref{gnksi} we see that  $g_1(n,k,s_1)$ is a decreasing function of $s_1$ and $s_1\geq t+3$.
Thus, for $n\geq (t+2)^2 k^2+t+1$ we have
\begin{align*}
|\hf||\hg|&\leq (t+1)g_1(n,k,t+3)
\left((k-t+1)\binom{n-t-1}{k-t-1}+k^2\binom{n-t-2}{k-t-2}\right)\\[5pt]&\leq (t+1)\binom{t+3}{3}k^3\binom{n-t-3}{k-t-3}
\left((k-t+1)\binom{n-t-1}{k-t-1}+k^2\binom{n-t-2}{k-t-2}\right)\\[5pt]
&\leq \frac{(t+1)^2(t+2)(t+3)k^3 (k-t-1)^2}{6(n-t-1)^2}
\left((k-t+1)+\frac{k^2(k-t-1)}{n-t-1}\right)\binom{n-t-1}{k-t-1}^2\\[5pt]
&\leq \frac{(t+2)^4k^4 (k-t+1)}{6(n-t-1)^2}
\left(k-t+1+\frac{k^2(k-t-1)}{n-t-1}\right)\binom{n-t-1}{k-t-1}^2\\[5pt]
&\leq \frac{k-t+1}{6}
\left(k-t+1+k-t+1\right)\binom{n-t-1}{k-t-1}^2\\[5pt]
&= \frac{(k-t+1)^2}{3}\binom{n-t-1}{k-t-1}^2.
\end{align*}
Apply \eqref{ineq-key3} with $c=4$, we conclude that
\[
|\hf||\hg| < \frac{2(k-t+1)^2}{3} \binom{n-k-1}{k-t-1}^2< |\hh(n,k,t)|^2.
\]

\vspace{6pt}
{\noindent\bf Case 3. } $s_1= t+2$ and $s_2=t+1$.
\vspace{6pt}

 By  \eqref{ineq-case1hf} and \eqref{ineq-case1hg},  we obtain that
\begin{align}\label{ineq-case3hf}
|\hf^{(\geq t+3)}|\leq (t+1)k^{3}\binom{n-t-3}{k-t-3}
\end{align}
and
\begin{align}\label{ineq-case3hg}
|\hg^{(\geq t+2)}|\leq \binom{t+2}{t}k^2\binom{n-t-2}{k-t-2}\leq \frac{(t+2)^2}{2}k^2\binom{n-t-2}{k-t-2}.
\end{align}

Fix some $S\in \hb_1^{(t+2)}$. By Lemma \ref{lem3-1} (i), we see that $|A_2\cap S|\geq t$ for all $A_2\in \ha_2$. For any $T\in \binom{S}{t}$, define
\[
N(T) =\{x\colon T\cup\{x\}\in \ha_2\} \mbox{ and }
\Gamma=\left\{T\in \binom{S}{t}\colon |N(T)|\geq 4\right\}.
\]
By Lemma \ref{lem-5.1} we see that $|N(T)|\leq k-t+1$. If $|N(T)|\geq 4$, then we claim that $T\subset S'$ for all $S'\in \hb_1^{(t+2)}$. Indeed, if $|S'\cap T|<t$, then for every $x\in N(T)$, $|S'\cap (T\cup \{x\})|\geq t$ implies $|S'\cap T|=t-1$ and $x\in S'$. It follows that $|S'|\geq t-1+4=t+3$, a contradiction.

\vspace{6pt}
{\noindent\bf Subcase 3.1. } There exist $T,T'\in \Gamma$ such that $T\cup T'=S$.
\vspace{6pt}

Then we have $T'\cup T\subset S'$ for all $S'\in \hb_1^{(t+2)}$. It follows that $|\hb_1^{(t+2)}|=1$.
By  \eqref{ineq-hbub} we have $|\ha_2| \leq \binom{t+2}{t}(k-t+1)$.
By \eqref{ineq-case3hf} and \eqref{ineq-case3hg} we obtain that
\begin{align*}
|\hf||\hg|&\leq \left(|\hb_1^{(t+2)}|\binom{n-t-2}{k-t-2}+(t+1)k^3\binom{n-t-3}{k-t-3}\right)\\[5pt]
&\quad\quad\cdot\left(|\ha_2|\binom{n-t-1}{k-t-1}+\frac{(t+2)^2}{2}k^2\binom{n-t-2}{k-t-2}\right)\\[5pt]
&\leq \left(\frac{k-t-1}{n-t-1}+ \frac{(t+1)k^3 (k-t-1)^2}{(n-t-1)^2}\right)\binom{n-t-1}{k-t-1}\\[5pt]
&\quad\quad\cdot\left(\frac{(t+2)^2}{2}(k-t+1)+\frac{(t+2)^2k^2(k-t-1)}{2(n-t-1)}\right) \binom{n-t-1}{k-t-1}.
\end{align*}
Therefore, using $n\geq (t+2)^2k^2+t+1$ and $k\geq 5$ we arrive at
\begin{align*}
|\hf||\hg|&\leq \left(\frac{(t+2)^2(k-t-1)}{(n-t-1)}+ \frac{(t+1)(t+2)^2k^3 (k-t-1)^2}{(n-t-1)^2}\right)\\[5pt]
&\quad\quad\cdot\left(\frac{k-t+1}{2}+\frac{k^2(k-t-1)}{2(n-t-1)}\right) \binom{n-t-1}{k-t-1}^2\\[5pt]
&\leq \left(\frac{k-t-1}{4}+ \frac{ k-t+1}{4}\right)\left(\frac{k-t+1}{2}+\frac{k-t+1}{4}\right) \binom{n-t-1}{k-t-1}^2\\[5pt]
&\leq \frac{3}{8}(k-t+1)^2 \binom{n-t-1}{k-t-1}^2.
\end{align*}
Apply \eqref{ineq-key3} with $c=4$, we conclude that
\[
|\hf||\hg| < \frac{3(k-t+1)^2}{4} \binom{n-k-1}{k-t-1}^2< |\hh(n,k,t)|^2.
\]

\vspace{6pt}
{\noindent\bf Subcase 3.2. } $T\cup T'\subsetneq S$ for every  $T,T'\in \Gamma$.
\vspace{6pt}

Note that $T\cup T'\subsetneq S$ implies $|T\cap T'|=t-1$ for every  $T,T'\in \Gamma$. Fix some $T,T'\in \Gamma$ and let $C=T\cap T'$, $T\setminus C= \{x\}$, $T'\setminus C= \{y\}$. Recall that $\Gamma \subset \binom{S}{t}$ and $|S|=t+2$. If there exists $T''\in \Gamma\setminus \{T,T'\}$ such that $C\subset T''$, then we infer that $|\Gamma|=3$. If $|T''\cap C|\leq t-1$ for all $T''\in \Gamma\setminus \{T,T'\}$, then  $|T''\cap T|=t-1$ and $|T''\cap T'|=t-1$ imply that $x,y\in T''$ and $|T''\cap C|=t-1$. Hence there are at most $t-1$ possibilities for $T''$ and $|\Gamma|\leq t+1$.  Thus $|\Gamma|\leq t+1$ and it follows that
\begin{align}\label{ineq-subcase3.2ha2}
|\ha_2|\leq |\Gamma|(k-t+1)+3\left(\binom{t+2}{t}-|\Gamma|\right)\leq (t+1)(k-t-2)+3\binom{t+2}{2}.
\end{align}
By \eqref{ineq-case1hf} we have
\begin{align}\label{ineq-subcase3.2hf}
|\hf|\leq (t+1)k^2\binom{n-t-2}{k-t-2}.
\end{align}
Therefore, by \eqref{ineq-case3hg}, \eqref{ineq-subcase3.2ha2} and \eqref{ineq-subcase3.2hf}
we obtain that
\begin{align*}
|\hf||\hg|&\leq (t+1)k^2\binom{n-t-2}{k-t-2}
\left(|\ha_2|\binom{n-t-1}{k-t-1}+\frac{(t+2)^2}{2}k^2\binom{n-t-2}{k-t-2}\right)\\[5pt]
&\leq (t+1)k^2\binom{n-t-2}{k-t-2}
\left((t+1)(k-t-2)+3\binom{t+2}{2}\right)\binom{n-t-1}{k-t-1}\\[5pt]
&\qquad+(t+1)k^2\binom{n-t-2}{k-t-2}\frac{(t+2)^2}{2}k^2\binom{n-t-2}{k-t-2}\\[5pt]
&\leq \left(\frac{(t+1)^2k^2(k-t-1)^2}{n-t-1}+\frac{3(t+1)^2(t+2)k^2(k-t-1)}{2(n-t-1)}\right)\binom{n-t-1}{k-t-1}^2\\[5pt]
&\qquad+\frac{(t+1)(t+2)^2k^4(k-t-1)^2}{2(n-t-1)^2}\binom{n-t-1}{k-t-1}^2.
\end{align*}
For $n\geq 4(t+2)^2k^2$, we arrive at
\begin{align*}
|\hf||\hg|&< \left(\frac{(k-t+1)^2}{4}+\frac{3(t+1)(k-t+1)}{8}+\frac{(k-t+1)^2}{8}\right)
\binom{n-t-1}{k-t-1}^2\\[5pt]
&\leq \frac{3}{4}\max\left\{(t+2)^2,(k-t+1)^2\right\}\binom{n-t-1}{k-t-1}^2.
\end{align*}
Now apply \eqref{ineq-key3} with $c=8$, we conclude that
\[
|\hf||\hg| < \max\left\{(t+2)^2,(k-t+1)^2\right\}\binom{n-k-1}{k-t-1}^2< \max\{|\ha(n,k,t)|^2,|\hh(n,k,t)|^2\}.
\]

\vspace{6pt}
{\noindent\bf Case 4. } $s_1= t+1$ and $s_2=t+1$.
\vspace{6pt}

Recall that $\ha_1=\hht_t^{(t+1)}(\hg)$ and $\ha_2=\hht_t^{(t+1)}(\hf)$. By the assumption, we have $\ha_1\neq \emptyset\neq \ha_2$. By \eqref{ineq-case1hf} and \eqref{ineq-case1hg},
 we have
\begin{align}\label{ineq-hf1}
|\hf|\leq |\ha_1|\binom{n-t-1}{k-t-1}+(t+1)k^2\binom{n-t-2}{k-t-2}
\end{align}
and
\begin{align}\label{ineq-hg1}
|\hg|\leq |\ha_2|\binom{n-t-1}{k-t-1}+(t+1)k^2\binom{n-t-2}{k-t-2}.
\end{align}
 By Proposition 3.1, we may assume that $\ha_1\cap \ha_2=\emptyset$.
Then \eqref{ineq-hbub2} implies $|\ha_1|+|\ha_2|\leq 4(k-t+1)$.

If $|\ha_1||\ha_2|\leq \max\{(t+2)^2,(k-t+1)^2\}/2$, then by multiplying \eqref{ineq-hf1}
and \eqref{ineq-hg1} we get
\begin{align*}
|\hf||\hg|&\leq |\ha_1||\ha_2|\binom{n-t-1}{k-t-1}^2
+(t+1)k^2\binom{n-t-2}{k-t-2}\binom{n-t-1}{k-t-1}(|\ha_1|+|\ha_2|)\\[5pt]
&\quad\quad +(t+1)^2k^4\binom{n-t-2}{k-t-2}^2\\[5pt]
&\leq \frac{1}{2}\max\left\{(t+2)^2,(k-t+1)^2\right\}\binom{n-t-1}{k-t-1}^2
+\frac{4(t+1)k^2(k-t+1)^2}{n-t-1}\binom{n-t-1}{k-t-1}^2\\[5pt]
&\quad\quad +\frac{(t+1)^2k^4(k-t-1)^2}{(n-t-1)^2}\binom{n-t-1}{k-t-1}^2.
\end{align*}
For $n\geq 4(t+2)^2k^2$, we arrive at
\begin{align*}
|\hf||\hg|&\leq \left(\frac{1}{2}\max\left\{(t+2)^2,(k-t+1)^2\right\}
+\frac{(k-t+1)^2}{4}+\frac{(k-t+1)^2}{8}\right)\binom{n-t-1}{k-t-1}^2\\[5pt]
&\leq \frac{7}{8}\max\left\{(t+2)^2,(k-t+1)^2\right\}\binom{n-t-1}{k-t-1}^2.
\end{align*}
Applying \eqref{ineq-key3} with $c=16$,  we obtain that
\begin{align*}
|\hf||\hg| &\leq  \frac{7}{8}\times \frac{16}{16-2}\max\left\{(t+2)^2,(k-t+1)^2\right\}\binom{n-k-1}{k-t-1}^2\\[5pt]
&=\max\left\{(t+2)^2,(k-t+1)^2\right\}\binom{n-k-1}{k-t-1}^2\\[5pt]
&\leq \max\left\{|\ha(n,k,t)|^2,|\hh(n,k,t)|^2\right\}.
\end{align*}
Thus in the rest of the proof we may assume that
\begin{align}\label{indirectAssump}
|\ha_1||\ha_2|> \frac{1}{2}\max\left\{(t+2)^2,(k-t+1)^2\right\}.
 \end{align}

Since $\ha_1\cap \ha_2=\emptyset$, we see that $\ha_1,\ha_2$ are non-empty exact cross $t$-intersecting.
Applying  Proposition \ref{prop-1.4} with $\ha=\ha_1$, $\hb=\ha_2$,
we see that one of (i), (ii), (iii) in Proposition \ref{prop-1.4}  holds.
If (iii) holds, then $|\ha_1||\ha_2|\leq \frac{(t+2)^2}{2}$, contradicting \eqref{indirectAssump}.
Thus, either (i) or (ii) of Proposition \ref{prop-1.4}  holds.

\vspace{6pt}
{\noindent\bf Subcase 4.1.} Either $|\ha_1|\leq 2$, $|\ha_2|\leq k+1$, $\tau_t(\ha_2)\geq  t+1$
or $|\ha_2|\leq 2$, $|\ha_1|\leq k+1$, $\tau_t(\ha_1)\geq t+1$.
\vspace{6pt}

By symmetry, assume that $|\ha_1|\leq 2$, $|\ha_2|\leq k+1$ and  $\tau_t(\ha_2)\geq  t+1$.
Let $r_2$ be the minimum integer such that $\tau_{t}(\hb_2^{(\leq r_2)})\geq t+1$.
Then clearly $s_2=r_2=t+1$. Let $\alpha_\ell' = \left((t+1)^2k^{\ell-t-1}\right)^{-1}|\hb_1^{(\ell)}|$.
 By \eqref{ineq-crosshb} we have
\begin{align}\label{ineq-alhpa2}
\sum\limits_{t+2\leq \ell\leq k}\alpha_\ell' \leq \sum\limits_{t+1\leq \ell\leq k}\alpha_\ell' \leq 1.
\end{align}
Let $h(n,k,\ell)=(t+1)^2k^{\ell-t-1}\binom{n-\ell}{k-\ell}$. Since
\[
\frac{h(n,k,\ell+1)}{h(n,k,\ell)} =\frac{k^{\ell-t}\binom{n-\ell-1}{k-\ell-1}}{k^{\ell-t-1}
\binom{n-\ell}{k-\ell}} =\frac{k(k-\ell)}{n-\ell}< 1,
\]
by \eqref{ineq-alhpa2} we infer
\[
\sum_{t+2\leq \ell\leq k} |\hf^{(\ell)}| \leq \sum_{t+2\leq \ell\leq k} \alpha_\ell' h(n,k,\ell)
\leq h(n,k,t+2)=(t+1)^2k\binom{n-t-2}{k-t-2}.
\]
It follows that
\begin{align}\label{ineq-hf2}
|\hf| =|\hf^{(t+1)}|+\sum_{t+2\leq \ell\leq k} |\hf^{(\ell)}|
&\leq |\ha_1|\binom{n-t-1}{k-t-1}+(t+1)^2k\binom{n-t-2}{k-t-2}\nonumber\\[5pt]
&\leq 2\binom{n-t-1}{k-t-1}+(t+1)^2k\binom{n-t-2}{k-t-2}.
\end{align}

If $k=t+1$ then
\[
|\hf||\hg|=|\ha_1||\ha_2|\leq 2(k+1)\leq 2(t+2)<|\ha(n,t+1,t)|^2
 \]
 and we are done. Hence we may assume that $k\geq t+2$.  Note that $k\geq t+2$ and $k\geq 5$ imply $2(k+1)\leq (t+1)(k-t+1)$. Therefore, by \eqref{ineq-hf2} and  \eqref{ineq-hg1} we obtain that
\begin{align*}
|\hf||\hg|&\leq\left(2\binom{n-t-1}{k-t-1}+(t+1)^2k\binom{n-t-2}{k-t-2}\right)\\[5pt]
&\quad\quad\cdot\left((k+1)\binom{n-t-1}{k-t-1}+(t+1)k^2\binom{n-t-2}{k-t-2}\right)\\[5pt]
&\leq 2(k+1)\binom{n-t-1}{k-t-1}^2+\left(2(t+1)k^2+2(t+1)^2k(k+1)\right)\binom{n-t-2}{k-t-2}\binom{n-t-1}{k-t-1}\\[5pt]
&\quad\quad +(t+1)^3k^3\binom{n-t-2}{k-t-2}^2.
\end{align*}
Note that $(t+1)(k+1)=tk+k+t+1\leq tk+2k=(t+2)k$ and thereby
\[
2(t+1)k^2+2(t+1)^2k(k+1)\leq 2(t+1)k^2+2(t+1)(t+2)k^2=2(t+1)(t+3)k^2.
\]
Then for $n\geq 4(t+2)^2k^2$,
\begin{align*}
|\hf||\hg|&< (t+1)(k-t+1)\binom{n-t-1}{k-t-1}^2+\frac{2(t+1)(t+3)k^2(k-t-1)}{n-t-1}\binom{n-t-1}{k-t-1}^2\\[5pt]
&\quad\quad +\frac{(t+1)^3k^3(k-t-1)^2}{(n-t-1)^2}\binom{n-t-1}{k-t-1}^2\\[5pt]
&< (t+1)(k-t+1)\binom{n-t-1}{k-t-1}^2+\frac{k-t-1}{2}\binom{n-t-1}{k-t-1}^2 +\frac{k-t-1}{4}\binom{n-t-1}{k-t-1}^2\\[5pt]
&= \left(t+\frac{7}{4}\right)(k-t+1)\binom{n-t-1}{k-t-1}^2.
\end{align*}
Now applying \eqref{ineq-key3} with $c=8(t+2)$, we conclude that
\begin{align*}
|\hf||\hg|< \left(t+\frac{7}{4}\right)(k-t+1) \frac{8(t+2)}{8(t+2)-2}\binom{n-k-1}{k-t-1}^2&=(t+2)(k-t+1)\binom{n-k-1}{k-t-1}^2\\[5pt]
&\leq  \max\left\{|\ha(n,k,t)|^2,|\hh(n,k,t)|^2\right\}.
\end{align*}

\vspace{6pt}
{\noindent\bf Subcase 4.2.} $\ha_1\cup \ha_2$ is a sunflower with center of size $t$.
\vspace{6pt}

Without loss of generality, assume that $\ha_1=\{[t]\cup \{a_1\},\ldots, [t]\cup \{a_p\} \}$ and $\ha_2=\{[t]\cup \{b_1\},\ldots, [t]\cup \{b_q\}\}$.  By Lemma \ref{lem-5.1} we see $p\leq k-t+1$ and $q\leq k-t+1$. If $p\leq 2$ or $q\leq 2$ holds, then
\[
|\ha_1||\ha_2|\leq 2(k-t+1)\leq \frac{1}{2}(t+2)(k-t+1)\leq \frac{1}{2} \max\left\{(t+2)^2,(k-t+1)^2\right\},
\]
contradicting \eqref{indirectAssump}. Thus we further assume that $p,q\geq 3$. Let
$\hf_0= \hf([t])$, $\hf_1=\hf\setminus \hf_0$, $\hg_0=\hg([t])$ and $\hg_1=\hg\setminus \hg_0$.

\begin{claim}\label{claim1}
For each $F\in \hf_1$, $|F\cap[t]|= t-1$ and $\{b_1,\ldots,b_q\}\subset F$. Similarly, for each $G\in \hg_1$, $|G\cap[t]|= t-1$ and $\{a_1,\ldots,a_p\}\subset G$.
\end{claim}
\begin{proof}
Indeed, simply note that $|F\cap ([t]\cup \{b_j\})|\geq t$ for $j=1,\ldots,q$ and $[t]\not\subset F$, we see that $|F\cap[t]|= t-1$ and $\{b_1,\ldots,b_q\}\subset F$.
\end{proof}

By Claim \ref{claim1}, we see that
\begin{align}\label{ineq-case4.2hf1hg1}
|\hf_1|\leq t\binom{n-t-q+1}{k-t-q+1}\leq t\binom{n-t-2}{k-t-2},\ |\hg_1|\leq t\binom{n-t-p+1}{k-t-p+1}\leq t\binom{n-t-2}{k-t-2}.
\end{align}

By non-triviality, we know that $\hf_1\neq \emptyset\neq \hg_1$. Fix some $F\in \hf_1$.
Then for every $G_0\in \hg_0$, since $|G_0\cap F|\geq t$ and $|F\cap [t]|= t-1$,
$G_0\cap (F\setminus [t])\neq \emptyset$. Therefore,
\begin{align}\label{ineq-case4.2hg0}
|\hg_0|\leq \binom{n-t}{k-t}-\binom{n-t-|F\setminus [t]|}{k-t}= \binom{n-t}{k-t}-\binom{n-k-1}{k-t}=|\hh(n,k,t)|-t.
\end{align}
Similarly,
\begin{align}\label{ineq-case4.2hf0}
|\hf_0|\leq \binom{n-t}{k-t}-\binom{n-k-1}{k-t}=|\hh(n,k,t)|-t.
\end{align}

Let $\hk_1=\{F\setminus [t]\colon F\in \hf_1\}$ and $\hk_2=\{G\setminus [t]\colon G\in \hg_1\}$.
If $|\hk_1|=|\hk_2|=1$, then Claim \ref{claim1} implies that $|\hf_1|,|\hg_1|\leq t$.
It follows that $|\hf||\hg| \leq |\hh(n,k,t)|^2$ and we are done.
If at least one of $\hk_1$ and $\hk_2$ has size greater than one, without loss of generality,
assume that $|\hk_1|\geq 2$ and let $K_1,K_2\in \hk_1$. Then $F\cap K_i\neq \emptyset$ for $i=1,2$
and for each $F\in \hf_0$. Therefore
$F\cap (K_1\cap K_2)\neq \emptyset$ or $F\cap (K_1\setminus K_2)\neq \emptyset$,
$F\cap (K_2\setminus K_1)\neq \emptyset$. It follows that
\[
|\hf_0|\leq \binom{n-t}{k-t}-\binom{n-t-|K_1\cap K_2|}{k-t} +|K_2\setminus K_1||K_1\setminus K_2|\binom{n-t-2}{k-t-2}.
\]
Let $|K_2\setminus K_1|=x$. Note that $|K_1|=|K_2|=k-t+1$ implies that
$1\leq x\leq k-t$. Define
\[
\varphi(x) =\binom{n-t}{k-t}-\binom{n-k-1+x}{k-t} +x^2\binom{n-t-2}{k-t-2}.
\]
Note that for $n\geq 4k^2$,
\begin{align*}
\varphi'(x) &=-\binom{n-k-1+x}{k-t}\sum_{i=0}^{k-t-1}\frac{1}{n-k-1+x-i} +2x\binom{n-t-2}{k-t-2}\\[5pt]
&\leq -\binom{n-k}{k-t}\frac{k-t}{n-t} +2(k-t)\binom{n-t-2}{k-t-2}\\[5pt]
&\overset{\eqref{ineq-key}}{\leq} -\frac{n-t-(k-t)^2}{n-t}\binom{n-t}{k-t}\frac{k-t}{n-t}
+2(k-t)\binom{n-t-2}{k-t-2}\\[5pt]
&\leq -\frac{1}{2}\binom{n-t-1}{k-t-1}+2(k-t)\binom{n-t-2}{k-t-2}\\[5pt]
&<0.
\end{align*}
By $x\geq 1$, we infer
\begin{align}\label{ineq-case4.2hf00}
|\hf_0|\leq \varphi(1)&=\binom{n-t}{k-t}-\binom{n-k}{k-t} +\binom{n-t-2}{k-t-2}\nonumber\\[5pt]
&= |\hh(n,k,t)| -\binom{n-k-1}{k-t-1} +\binom{n-t-2}{k-t-2}.
\end{align}
Thus by \eqref{ineq-case4.2hf00}, \eqref{ineq-case4.2hf1hg1} and \eqref{ineq-case4.2hg0} we have
\begin{align}\label{ineq-hfhg}
|\hf||\hg|&=\left(|\hf_0|+|\hf_1|\right)\left(|\hg_0|+|\hg_1|\right)\nonumber\\[5pt]
&\leq \left(|\hh(n,k,t)| -\binom{n-k-1}{k-t-1}+(t+1)\binom{n-t-2}{k-t-2}\right)\nonumber\\[5pt]
&\qquad\qquad\cdot\left(|\hh(n,k,t)| +t\binom{n-t-2}{k-t-2}\right)\nonumber\\[5pt]
&= |\hh(n,k,t)|^2 + |\hh(n,k,t)|\left((2t+1)\binom{n-t-2}{k-t-2}-\binom{n-k-1}{k-t-1}\right)\nonumber\\[5pt]
&\qquad\qquad+ t(t+1)\binom{n-t-2}{k-t-2}^2-t\binom{n-k-1}{k-t-1}\binom{n-t-2}{k-t-2}.
\end{align}

Applying \eqref{ineq-key2} with $c=2$, we have
\[
\binom{n-k-1}{k-t-1}\geq \frac{c-1}{c} \binom{n-t-1}{k-t-1}=\frac{1}{2} \binom{n-t-1}{k-t-1}.
\]
Then for $n\geq 4(t+1)k$,
\begin{align}\label{ineq-case4.2part1}
(2t+1)\binom{n-t-2}{k-t-2}-\binom{n-k-1}{k-t-1}& \leq (2t+1)\binom{n-t-2}{k-t-2}-\frac{1}{2}\binom{n-t-1}{k-t-1}\nonumber\\[5pt]
&= \binom{n-t-2}{k-t-2}\left(2t+1-\frac{n-t-1}{2(k-t-1)}\right)\nonumber\\[5pt]
&<0,
\end{align}
and
\begin{align}\label{ineq-case4.2part2}
 & t(t+1)\binom{n-t-2}{k-t-2}^2-t\binom{n-k-1}{k-t-1}\binom{n-t-2}{k-t-2}\nonumber\\[5pt]
 \leq  & t(t+1)\binom{n-t-2}{k-t-2}^2-\frac{t}{2}\binom{n-t-1}{k-t-1}\binom{n-t-2}{k-t-2}\nonumber\\[5pt]
  =  & t\binom{n-t-2}{k-t-2}^2\left(t+1-\frac{(n-t-1)}{2(k-t-1)}\right)\nonumber\\[5pt]
  <&0.
\end{align}
By \eqref{ineq-hfhg}, \eqref{ineq-case4.2part1} and \eqref{ineq-case4.2part2},
we conclude that $|\hf||\hg|<|\hh(n,k,t)|^2$ and the theorem is proven.
\end{proof}


{\noindent \bf Acknowledgement. }We would like to thank the referees for their helpful comments and detailed corrections. The second author was
supported by the National Natural Science Foundation of China (No. 11701407).

\end{document}